\documentclass[a4paper,11pt]{article}

\usepackage{amsmath,amsthm,amssymb} 
\usepackage{hyperref}
\usepackage[applemac]{inputenc}

\usepackage{fourier}

\usepackage[T1]{fontenc}

\usepackage{geometry}

\DeclareMathOperator{\im}{Im}

\DeclareMathOperator{\Id}{Id}

\DeclareMathOperator{\tr}{Tr}

\DeclareMathOperator{\diam}{diam}

\DeclareMathOperator{\Span}{Span}

\DeclareMathOperator{\SO}{SO}

\DeclareMathOperator{\PSL}{PSL}

\DeclareMathOperator{\Rm}{Rm}

\DeclareMathOperator{\vol}{vol}

\newcommand{\R}{\mathbb R}
\newcommand{\C}{\mathbb C}
\newcommand{\Z}{\mathbb Z}

\newcommand{\diff}{\text{\rm d}}
\newcommand{\del}{\partial}
\newcommand{\delb}{\bar{\del}}

\newcommand{\so}{\mathfrak{so}}

\renewcommand{\im}{\mathrm {im}\,}
\renewcommand{\P}{\mathbb P}

\theoremstyle{plain}
	\newtheorem{theorem}{Theorem}
	\newtheorem{proposition}[theorem]{Proposition}
	\newtheorem{lemma}[theorem]{Lemma}
	\newtheorem{corollary}[theorem]{Corollary}

\theoremstyle{definition}
	\newtheorem{definition}[theorem]{Definition}
	\newtheorem{remark}[theorem]{Remark}

\theoremstyle{plain}
	\newtheorem*{theorem*}{Theorem}
	\newtheorem*{proposition*}{Proposition}
	\newtheorem*{lemma*}{Lemma}
	\newtheorem*{corollary*}{Corollary}
	\newtheorem*{conjecture*}{Conjecture}
\theoremstyle{definition}
	\newtheorem*{definition*}{Definition}
	\newtheorem*{remark*}{Remark}
	\newtheorem*{remarks*}{Remarks}
	
\begin{document}

\title{Limits of Riemannian 4-manifolds and the symplectic geometry of their twistor spaces}

\author{Joel Fine\thanks{Département de mathématiques, Université libre de Bruxelles, Brussels 1050, Belgium.}}

\maketitle

\begin{abstract}The twistor space of a Riemannian 4-manifold carries two almost complex structures, $J_+$ and $J_-$, and a natural closed 2-form $\omega$. This article studies limits of manifolds for which $\omega$ tames either $J_+$ or $J_-$. This amounts to a curvature inequality involving self-dual Weyl curvature and Ricci curvature, and which is satisfied, for example, by all anti-self-dual Einstein manifolds with non-zero scalar curvature. We prove that if a sequence of manifolds satisfying the curvature inequality converges to a hyperkähler limit $X$ (in the $C^2$ pointed topology) then $X$ cannot contain a holomorphic 2-sphere (for any of its hyperkähler complex structures). In particular, this rules out the formation of bubbles modelled on ALE gravitational instantons in such families of metrics.  
\end{abstract}
%
%
%

\section{Introduction}

\subsection{Statement of the main result}

We begin by briefly recalling the relevant parts of twistor theory, referring to the references for details. Let $(M,g)$ be an oriented Riemannian 4-manifold. The twistor space of $M$ is the fibre bundle $\pi \colon Z \to M$ whose fibre over $x\in M$ is the 2-sphere of all almost complex structures on $T_xM$ which are orthogonal with respect to $g$ and compatible with the orientation on $M$. 

There are two almost complex structures, $J_+$ and $J_-$, on $Z$. The first, $J_+$, was introduced by Atiyah--Hitchin--Singer \cite{Atiyah:1978aa}, following Penrose \cite{Penrose:1967aa}. The second, $J_-$, was defined by Eells--Salamon \cite{Eells:1985aa}. To describe them note that the Levi-Civita connection of $(M,g)$ gives a hori\-zontal-vertical splitting of the tangent bundle to $Z$ which we write as $TZ = V \oplus H$. This splitting is invariant with respect to both $J_{\pm}$; on $V$, $J_\pm$ is equal to the usual complex structure of the 2-sphere; at $p \in Z$ we take $J_+$ to be equal to the complex structure on $T_{\pi(p)}M$ determined by $p$, under the identification $\pi_*\colon H_p \to T_{\pi(p)}M$. Taking the opposite sign on $H$ in this definition gives $J_-$. We remark in passing that the Atiyah--Hitchin--Singer almost complex structure $J_+$ is integrable when the metric is anti-self-dual \cite{Atiyah:1978aa}; by contrast $J_-$ is never integrable.

In \cite{Reznikov:1993aa} Reznikov observed that $Z$ also carries a natural closed 2-form $\omega$, whose restriction to each fibre of $\pi$ is the area form. Asking for $\omega$ to be symplectic gives a curvature inequality for $g$ which was first investigated by Reznikov \cite{Reznikov:1993aa} and later described explicitly in \cite{Fine:2009aa}. (In fact, this can be seen as a special case of the ``fat connections'' introduced much earlier by Weinstein \cite{Weinstein:1980aa}.) In this article we will focus on the case when $\omega$ tames $J_+$ or $J_-$. Recall that a 2-form $\omega$ tames an almost complex structure $J$ if it is positive on every $J$-complex line; when this happens $\omega$ is automatically non-degenerate. Asking for $\omega$ to tame either $J_+$ or $J_-$ gives another, stronger, curvature inequality, also described in \cite{Fine:2009aa} and which we now recall here.

To do so, we need the decomposition of a 4-dimensional curvature tensor. Write the curvature operator $\Rm \colon \Lambda^2 \to \Lambda^2$ in block form with respect to the decomposition $\Lambda^2 = \Lambda^+ \oplus \Lambda^-$ into self-dual and anti-self-dual forms:
\begin{equation}\label{curvature-decomposition}
\Rm = \left(
\begin{array}{cc}
A& B^*\\ B&  C
\end{array}
\right)
\end{equation}
where $A \colon \Lambda^+ \to \Lambda^+$, $B \colon \Lambda^+ \to \Lambda^-$ and $C \colon \Lambda^- \to \Lambda^-$. To relate this to the usual curvature decomposition, one identifies a trace-free endomorphism of $TM$ with an infinitesimal change of conformal class and hence a linear map $\Lambda^+ \to \Lambda^-$ giving the infinitesimal change of the bundle of self-dual 2-forms. Under this procedure, $B \colon \Lambda^+ \to \Lambda^-$ is identified with the trace-free Ricci curvature. Meanwhile $\tr A = \tr C = R/4$, where $R$ is the scalar curvature, and the trace free parts of $A$ and $C$ are the self-dual and anti-self-dual parts of the Weyl curvature. (See \cite{Atiyah:1978aa} for more details.)

\begin{theorem}[Theorem~4.4 of \cite{Fine:2009aa}]
Let $g$ be a Riemannian metric on an oriented 4-manifold. Suppose that 
for all non-zero $\theta \in \Lambda^+$, 
\begin{equation}
|\langle A(\theta), \theta \rangle | > |B(\theta)||\theta|
\label{twistor-tame-inequality}
\end{equation}
(and note that this implies that $A$ is invertible). There are two possibilities.
\begin{enumerate}
\item
If $\det(A)>0$ then the Reznikov 2-form $\omega$ tames $J_+$.
\item 
If $\det(A)<0$ then the Reznikov 2-form  $\omega$ tames $J_-$.
\end{enumerate} 
\end{theorem}

\begin{definition}
When a Riemannian 4-manifold satisfies inequality \eqref{twistor-tame-inequality} we say its \emph{twistor space is tamed}. 
\end{definition}

Important examples of metrics whose twistor spaces are tamed are given by anti-self-dual Einstein metrics whose scalar curvature $R$ is non-zero. For these metrics, $B=0$ and $A = \frac{R}{12}\Id$. The only possible compact examples with $R>0$ are the standard metrics on $S^4$ and $\C\P^2$ (a result of Hitchin \cite{Hitchin:1981aa}). Hyperbolic 4-manifolds and complex-hyperbolic surfaces (with the non-complex orientation) are also anti-self-dual and Einstein. These are the only \emph{known} compact examples with $R<0$. It is a long-standing open problem to decide if there exists a compact anti-self-dual Einstein metric with $R<0$ which is not locally homogeneous. If one allows orbifold singularities, there are \emph{many} examples, even with $R>0$. See for example \cite{Galicki:1988aa,Calderbank:2006aa}. Returning to smooth metrics, but dropping  compactness, there are many beautiful constructions of complete anti-self-dual Einstein metrics with $R<0$ which are not locally homogeneous. Even if one prescribes the geometry to be asymptotically hyperbolic or complex-hyperbolic, such metrics come in infinite dimensional families. See \cite{Calderbank:2004aa,Biquard:2000aa,Biquard:2002aa} for details. Another class of metrics with tamed twistor spaces are those with sectional curvatures which are pointwise $2/5$-pinched. When the curvature is positive, $J_+$ is tamed; when the curvature is negative $J_-$ is tamed (see Remark~3.8 of~\cite{Fine:2009aa}). 

When the twistor space of $g$ is tamed, one can use symplectic geometry and the theory of $J$-holomorphic curves in $(Z,\omega)$ to study the Riemannian geometry of $(M,g)$. The point of this article is to explore how this symplectic approach combines with the theory of convergence of Riemannian manifolds. Our main result rules out certain locally hyperkähler metrics arising as limits of metrics with tamed twistor spaces. (Recall that a locally hyperkähler metric is one whose universal cover is hyperkähler.)

\begin{theorem}\label{main-result}
Let $(M_i,g_i)$ be a sequence of Riemannian 4-manifolds with tamed twistor spaces. Suppose the sequence converges to a locally hyperkähler limit $(X,g)$ in the pointed $C^2$-topology. Then the universal cover of $X$ cannot contain a holomorphic 2-sphere (for any of its hyperkähler complex structures).
\end{theorem}

Note that locally hyperkähler metrics lie in the maximally degenerate ``boundary component'' of inequality \eqref{twistor-tame-inequality}, with $A = B = 0$. In this case, at least on the universal cover, $\omega$ is the pull-back of the area form on $S^2$ via the projection $Z \to S^2$ given by the hyperkähler complex structures, and so is visibly degenerate.

\subsection{Applications}

Before we give the proof in \S\ref{proof-of-main-result}, we first discuss some potential applications. There are certain situations in which the well-established theory of convergence shows that the only way a family of Riemannian manifolds can fail to be compact is when they bubble off an ALE gravitational instanton. These bubbles automatically contain a holomorphic $-2$-curve (as is proved in the course of Kronheimer's classification \cite{Kronheimer:1989aa}). It follows that when we add the hypothesis that the family of metrics all have tamed twistor spaces, Theorem~\ref{main-result} allows us to conclude the family is actually compact. 

As a concrete example, we give the following corollary of our main result.

\begin{corollary}\label{compact-asdE}
Let $(M_i, g_i)$ be a sequence of compact oriented 4-manifolds with anti-self-dual Einstein metrics of non-zero scalar curvature, normalised so that $|R(g_i)|=1$. Suppose that 
\begin{enumerate}
\item
The diameter $\diam(M_i,g_i) \leq D$ is uniformly bounded.
\item
The total volume $\vol(M_i, g_i) \geq V >0$ is uniformly bounded away from zero.
\end{enumerate}
Then a subsequence converges in the $C^\infty$ topology to an anti-self-dual Einstein metric $(M,g)$ of non-zero scalar curvature. 
\end{corollary}

The point is that under these conditions, work of Anderson, Bando--Kasue--Nakajima and Cheeger--Naber shows that a sequence of Einstein 4-manifolds converges to an \emph{orbifold} with isolated singularities \cite{Anderson:1990aa,Nakajima:1988aa,Bando:1989aa,Cheeger2015Regularity-of-E}.  Moreover, rescaling near the orbifold points one sees the singularity is modelled on an ALE Ricci-flat space which is not simply flat. Now the additional hypothesis that $g_i$ be anti-self-dual ensures that the ALE model is actually locally hyperkähler, with universal cover an ALE gravitational instanton. From here, Corollary~\ref{compact-asdE} follows from Theorem~\ref{main-result} and the fact that all non-flat ALE gravitational instantons contain $-2$-curves \cite{Kronheimer:1989aa}.

The condition in Corollary~\ref{compact-asdE} that the manifolds $(M_i, g_i)$ be compact is not necessary to apply the general convergence theory, it just makes for a clean statement. The same ideas apply to, for example, anti-self-dual Poincaré--Einstein metrics. These are complete anti-self-dual Einstein metrics on the interior of a compact manifold with boundary and which are asymptotically hyperbolic as one approaches the boundary. As mentioned above, such metrics come in infinite dimensional families. Theorem~\ref{main-result} ensures that these families cannot develop isolated orbifold singularities modelled on gravitational instantons. 

Corollary~\ref{compact-asdE} is related to work of Biquard~\cite{Biquard:2013aa}, which considers a family $(M_i, g_i)$ of Einstein 4-manifolds which converges to an orbifold $(M,g)$ with an isolated $\Z_2$-orbifold point at $p$. He makes an additional assumption that for large $i$, $(M_i,g_i)$ is well approximated by gluing in an appropriately scaled copy of the Eguchi--Hanson metric at $p$. In this situation, Biquard proves that at $p$, the part of the Riemann curvature operator of the limit metric $g$ which maps $\Lambda^+ \to \Lambda^+$ must  have a kernel. (The kernel corresponds to the complex structure at $p$ used to glue the Eguchi--Hanson metric there.) This can be used to prove an analogue of Corollary~\ref{compact-asdE}: if, in addition to Biquard's hypotheses we assume that the $(M_i,g_i)$ are anti-self-dual, with $|R(g_i)|=1$, then the limit metric $(M,g)$ is also anti-self-dual with $|R(g)|=1$. But then the curvature map $\Lambda^+ \to \Lambda^+$ is a non-zero multiple of the identity and so no such orbifold singularity occurs. 

There are two ways in which Corollary~\ref{compact-asdE} is more general than the sort of compactness result which can be obtained by arguing directly via Biquard's result. Firstly, \cite{Biquard:2013aa} considers only $\Z_2$-singularities resolved by the Eguchi--Hanson metric. This is most likely for simplicity of presentation; it seems certain that an analogous result could be proved for any resolution involving an ALE gravitational instanton. The second way is more significant. Biquard must assume that the degeneration occurs in a way tightly modelled by a specific gluing of Eguchi--Hanson. This is by no means guaranteed by the general convergence theory. 

Another situation in which one might hope to apply Theorem~\ref{main-result} is the following. It is an interesting open problem to construct a compact 4-manifold with an anti-self-dual Einstein metric of negative scalar curvature which is not simply a quotient of hyperbolic or complex-hyperbolic space. One might imagine attacking this with a continuity method: given a hyperbolic 4-manifold $(M,g)$ with a cone angle $2 \pi \beta$ along a surface $S \subset M$ (with $\beta <1$) is it possible to deform the metric, keeping it anti-self-dual and Einstein with negative scalar curvature, but ``opening out'' the cone, taking $\beta$ to 1? The crux to carrying out this plan is to control singularity formation as $\beta$ increases. Theorem~\ref{main-result} rules out isolated orbifold singularities forming in the part $M\setminus S$ of the manifold where the metric is smooth. 

\begin{remark}
The Riemannian convergence results mentioned apply more generally to manifolds with bounded Ricci curvature (as opposed to Einstein metrics), but the convergence is then $C^{1,\alpha}$ (as opposed to $C^\infty$). This is slightly too weak to apply Theorem~\ref{main-result}, which requires $C^2$-convergence. The problem in the proof given below occurs in Proposition~\ref{perturbing-regular-curves}. $C^{1,\alpha}$-convergence of the metrics gives only $C^{0,\alpha}$-converg\-ence of the twistor almost complex structures, whereas the proof of Proposition~\ref{perturbing-regular-curves} requires $C^1$-convergence. It would be interesting to know if an alternative argument can be used to overcome this seemingly technical shortfall.
\end{remark}

\subsection{Acknowledgements} 

I would like to thank Olivier Biquard, Gilles Carron, Claude LeBrun, Dmitri Panov and Misha Verbitsky for helpful conversations during the course of this work. I would also like to thank the anonymous referee for important observations on the first draft of this article. This research is supported by  ERC consolidator grant 646649 ``SymplecticEinstein'' and FNRS grant MIS F.4522.15.

\section{Proof of Theorem~\ref{main-result}}\label{proof-of-main-result}

In outline, the proof is simple. We give the sketch first and then fill in the details. First, pass to a subsequence so that the twistor spaces of $(M_i,g_i)$ all have either $J_+$ tamed or all have $J_-$-tamed. Write $Z$ for the twistor space of the limit $X$ and suppose for simplicity that $X$ is genuinely hpyerkähler (rather than just locally hyperkähler). The twistor spaces of the $g_i$ give a sequence $(J_i, \omega_i)$ of almost complex structures with taming symplectic forms on $Z$ which converge to the structures $(J, \omega)$ defined by the hyperkähler limit. (For a precise statement, see Lemma~\ref{twistor-spaces-converge} below.) Here $J$ is either $J_+$ or $J_-$, depending on the subsequence we chose and $\omega$ is the degenerate closed 2-form given by pulling back the area form from $S^2$ via the projection $Z \to S^2$. It will be important that all the forms $\omega_i$ lie in the same cohomology class as $\omega$.  

Now assume for a contradiction that there is a 2-sphere $S \subset X$ which is $I$-holomorphic for one of the hyperkähler complex structures $I$ on $X$. The horizontal copy of $S$ in $Z$, lying in the fibre of $Z \to S^2$ corresponding to $I$, is a $J$-holomorphic curve. Notice that the integral of $\omega$ over this lift is zero, since $\omega$ vanishes horizontally. One can show that this horizontal copy of $S$ is a \emph{regular} $J$-holomorphic curve, in the sense that the linearised Cauchy--Riemann operator is surjective. (See Corollary~\ref{spheres-are-regular} below.) The implicit function theorem can then be used to show that for large $i$ there is a $J_i$-holomorphic curve $S_i$, homotopic to the horizontal copy of $S$. (This is proved in Proposition~\ref{perturbing-regular-curves}.) Since $\omega_i$ tames $J_i$, it must have strictly positive integral over $S_i$. On the other hand, the $\omega_i$ are all cohomologous to $\omega$, so this integral vanishes for all $i$, giving a contradiction.

This section is organised as follows. In \S\ref{define-symplectic-form} we recall the definition of the closed 2-form on twistor space. We then explain how for a convergent sequence of Riemannian 4-manifolds, the closed 2-forms and almost complex structures on their twistor spaces converge. In \S\ref{deforming-regular-curves} we discuss the above mentioned application of the implicit function theorem, that given two nearby almost complex structures $J$ and $J'$, a regular $J$-holomorphic curve can be deformed to a $J'$-holomorphic curve. In \S\ref{regularity-proof} we show that a holomorphic 2-sphere in a hyperkähler manifold lifts to a \emph{regular} holomorphic 2-sphere in the twistor space. Finally, in \S\ref{details} we put the pieces together to complete the proof of Theorem~\ref{main-result}.

\subsection{From convergent metrics to convergent twistor spaces}\label{define-symplectic-form}

We begin by recalling the construction of the natural closed 2-form on twistor space. It is actually just as easy to consider a more general situation. Let $E\to M$ be an $\SO(3)$-vector bundle, $A$ a compatible connection and $Z \subset E$ the unit sphere bundle. We will define a closed 2-form on $Z$ whose restriction to each fibre is the area form. (This is a special case of a construction involving bundles of integral coadjoint orbits and which goes back at least as far as Weinstein's article \cite{Weinstein:1980aa}.) To do so we will generalise the following way of producing the area form on $S^2 \subset \R^3$. The tangent bundle is a sub-bundle of the trivial bundle: $TS^2 \subset S^2 \times \R^3$. The orthogonal projection of the product connection from $S^2 \times \R^3$ to $TS^2$ gives the Levi-Civita connection, whose curvature is equal to the area form $\omega_{S^2}$. 

We now carry out this same construction simultaneously on the fibres of $\pi \colon Z \to M$. Write $V = \ker \diff \pi \subset TZ$ for the vertical tangent bundle. We have $V \subset \pi^*E$ (just as $TS^2 \subset S^2 \times \R^3$) and we can use orthogonal projection of the connection $\pi^*A$ in $\pi^*E$ to produce a connection $\nabla$ in $V$. By construction, the restriction of $\nabla$ to each fibre of $\pi$ agrees with the Levi-Civita connection. Hence the curvature of $\nabla$ defines a closed 2-form $\omega$ on $Z$ whose restriction to each fibre is the area form. (Strictly speaking, $F_\nabla$ is an $\so(2)$-valued 2-form and to obtain a real 2-form we must orient the fibres of $V$.) Notice that, by Chern--Weil, $[\omega] = e(V)$ is the Euler class of $V \to Z$ (whose definition also requires that $V$ be oriented). 

Given  an oriented Riemannian 4-manifold $(M,g)$ we carry out this construction with $E = \Lambda^+$ and $A$ the Levi-Civita connection to obtain a closed 2-form $\omega$ on the twistor space $Z$. (This particular case of the construction was first considered explicitly by Reznikov \cite{Reznikov:1993aa}.) With this discussion behind us, we now consider what happens for convergent sequences of Riemannian manifolds. Let $(M_i, g_i)$ be a sequence of oriented Riemannian 4-manifolds which converge in the $C^2$ pointed topology to a limiting manifold $(X,g)$. Let $K \subset X$ be any compact subset. We fix attention momentarily on one of the twistor almost complex structures $J_+$ or $J_-$ (the argument being identical for either choice). Write $(Z, J, \omega) \to K$ and $(Z_i, J_{Z_i}, \omega_{Z_i}) \to M_i$ for the twistor spaces with their corresponding almost complex structures and closed 2-forms.

\begin{lemma}\label{twistor-spaces-converge}
In the situation of the previous paragraph, there exist maps $\phi_i \colon Z \to Z_i$, each of which is a diffeomorphism onto its image and covers a map $f_i \colon K \to M_i$, such that the sequence $(J_i, \omega_i) = \phi^*_i(J_{Z_i}, \omega_{Z_i})$ of structures on $Z$ satisfies the following.
\begin{enumerate}
\item
$J_i \to J$ in $C^1$.
\item
$[\omega_i] = [\omega]$ for all $i$.
\end{enumerate}
\end{lemma}
\begin{proof}
By definition of pointed convergence, there exists a sequence of maps $f_i \colon K \to M_i$, for which the metrics $h_i = f_i^*g_i$ converge in $C^2$ to $g$. Notice that the twistor space of $(K, h_i)$ is $f_i^*Z_i$. The map $\Lambda^2 \to \Lambda^+_{h_i}$ given by projecting against $\Lambda^-_{h_i}$ restricts to an isomorphism $\psi_i \colon \Lambda^+_{g} \to \Lambda^+_{h_i}$. Of course, $\psi_i$ is not an isometry and so the image under $\psi_i$ of the unit sphere bundle is not the twistor space of $h_i$, but by rescaling we obtain a diffeomorphism $\psi_i/|\psi_i|$ between $Z$ and the twistor space of $h_i$, and hence a fibrewise diffeomorphism $\phi_i \colon Z \to Z_i$ covering $f_i$.  

To prove $J_i \to J$, we first recall the definition of the twistor almost complex structures. The Levi-Civita connection of $g$ gives a vertical-horizontal decomposition $TZ = V \oplus H$; on $V$, $J$ is the standard complex structure on $S^2$, whilst on $H$, at the point $p \in Z$, $J$ is either plus or minus the almost complex structure on $H \cong T_{\pi(p)}Z$ corresponding to $p$, depending on whether we are talking about $J_+$ or $J_-$. 

The map $\psi_i \colon \Lambda_g^+ \to \Lambda_{h_i}^+$ depends algebraically on $h_i$, whilst the Levi-Civita connection is first order in $h_i$. Since $h_i \to g$ in $C^2$, the pull-back to $Z$ by $\psi_i/|\psi_i|$ of the Levi-Civita connection from the twistor space of $h_i$ converges in $C^1$ to the Levi-Civita connection of $g$. Moreover, on each fibre $\psi_i/|\psi_i|$ converges to the identity in $C^2$ and its derivative is $C^1$-close to being an isometry. This means that the pull-back of the twistor almost complex structure converges in $C^1$.

Finally note that $[\omega] = e(V)$, whilst $[\omega_i] = e(\phi_i^*V_i)$, where $V_i \to Z_i$ is the vertical tangent bundle. Since $\phi_i$ is a fibrewise diffeomorphism, its derivative in the vertical directions $\phi_* \colon V \to \phi_i^*V_i$ is an isomorphism and so $e(\phi_i^*V_i) =e(V)$ as claimed.
\end{proof}

\subsection{Deforming regular $J$-holomorphic curves}\label{deforming-regular-curves}

In this section we prove that regular $J$-holomorphic curves persist under deformation of the almost complex structure. This is a standard fact well-known to experts and variants of Proposition~\ref{perturbing-regular-curves} appear in many places in the literature. Unfortunately it has been impossible to track down exactly the version which is needed for the proof of Theorem~\ref{main-result} and so we give  it here. 
 
\begin{proposition}\label{perturbing-regular-curves}
Let $(Z,J)$ be an almost complex manifold and $u \colon (\Sigma,j) \to (Z,J)$ a regular $J$-holomorphic curve. There exists $\epsilon>0$ such that if $J'$ is another almost complex structure on $Z$ with $\|J - J'\|_{C^1}< \epsilon$ then there is a map $u' \colon \Sigma \to Z$, homotopic to $u$ and which is $J'$-holomorphic. (Here the $C^1$-norm is defined with respect to any auxiliary choice of metric on $Z$.)
\end{proposition}

\begin{proof}
We use the set-up described in McDuff--Salamon's book \cite{McDuff:2012aa}, which we refer to for more details. For any given almost complex structure $J$ on $Z$, we will define a smooth map $F_J$ of Banach spaces whose zeros correspond to $J$-holomorphic curves. To do this, first pick a connection $\nabla$ in $TZ$ and use $\nabla$-geodesics to define an exponential map $\exp \colon TZ \to Z$. Given a map $v \colon (\Sigma, j) \to Z$, we write $\delb_{J}(v) = \frac{1}{2}\left(\diff v + J(v) \circ \diff v \circ j\right)$. This is a section of $\Lambda^{0,1} \otimes v^*TZ$, where we use $J\circ v$ as the fibrewise complex structure in $v^*TZ$ in this tensor product. Now fix  $p>2$ (so that the Sobolev space of $W^{1,p}$ sections over $\Sigma$ embeds in $C^0$) and define a map
\begin{gather*}
F_J \colon W^{1,p}(\Sigma, u^*TZ) \to L^p(\Sigma, \Lambda^{0,1}\otimes u^*TZ)\\
F_J(\xi) = \Phi_\xi^{-1} \left(\delb_{J} (\exp \xi)\right)
\end{gather*}
Here $\Phi_\xi \colon u^*TZ \to \exp(\xi)^*TZ$ is the isomorphism given by parallel transport along the $\nabla$-geodesics joining $u$ and $\exp(\xi)$. One checks that $F_J$ is a smooth map of Banach spaces. (Strictly speaking, since we do not assume $Z$ is complete, we should restrict to domain of $F$ to a neighbourhood of the origin for which $\exp(\xi)$ makes sense.)

The result will be proved by finding $\xi$ such that $F_{J'}(\xi) = 0$, since then $\delb_{J'}(\exp \xi) = 0$ and, by elliptic regularity, $\exp \xi$ is the smooth $J'$-holomorphic curve we seek. To find such a $\xi$ we will show that $F_{J'}-F_{J}$ is controlled by $\|J - J'\|_{C^1}$. This, together with the hypothesis that $u$ is a regular $J$-holomorphic curve, will enable us to apply the implicit function theorem to $F_{J'}$ to find a zero. 

First note that $F_J$ and $F_{J'}$ take values in different spaces (since $u^*TZ$ has different fibrewise complex structures). To compare them, we will regard them as mapping into the larger space $L^p(\Lambda^1\otimes_\R u^*TZ)$. Direct from the definition we have
\begin{equation}\label{J-J'-difference}
(F_{J'} -F_J)(\xi) = \frac{1}{2}\Phi_{\xi}^{-1} \left((J' - J)(\exp \xi) \circ \diff (\exp \xi) \circ j\right)
\end{equation}
We restrict the domain of $F$ to a ball of radius $r$. There is a constant $K$ such that for $\|\xi\|_{W^{1,p}} < r$, we have $\| \diff (\exp \xi)\|_{L^p} \leq K$. It then follows from \eqref{J-J'-difference} that there is a constant $C$ such that all $\xi$ with $\| \xi\|_{W^{1,p}} < r$,
\begin{equation}\label{C0-control}
\| F_{J'}(\xi) - F_{J}(\xi) \|_{L^p} \leq C \| J' - J\|_{C^0}
\end{equation}

Next we differentiate \eqref{J-J'-difference} at $\xi$ in the direction of $\eta \in W^{1,p}(\Sigma, u^*TZ)$. The fact that $\xi$ appears in three places in \eqref{J-J'-difference} means there are three terms in the derivative. Schematically we can write it as
\begin{align*}
D_\xi(F_{J'} - F_J)(\eta)
	&=
		P_\xi(\eta) \big( (J'-J)(\exp \xi) \circ \diff(\exp \xi) \circ j \big)\\
	&\quad \quad
		+ \frac{1}{2}\Phi_\xi^{-1} \big( Q_\xi(\eta) \circ \diff (\exp \xi) \circ j \big)\\
	&\quad \quad \quad \quad
		+ \frac{1}{2} \Phi_\xi^{-1} \big( (J'-J)(\exp \xi) \circ R_{\xi}(\eta) \circ j \big)
\end{align*}
where $P_\xi(\eta)$ is the derivative of $\Phi_\xi^{-1}$, $Q_\xi(\eta)$ is the derivative of $(J'-J)(\xi)$ and $R_\xi(\eta)$ is the derivative of $\diff(\exp \xi)$, all taken at $\xi$ in the direction $\eta$. The terms involving $P_\xi$ and $R_\xi$ can be controlled with just the $C^0$-norm of $J'-J$, but $Q_\xi(\eta)$ sees the first derivative of $J'-J$. All together we see there is a constant $C$ such that for all $\xi$ with  $\|\xi\|_{W^{1,p}} < r$
\begin{equation}\label{C1-control}
\| D_\xi (F_{J'} - F_J)(\eta) \|_{L^p}
\leq
C \| J' - J \|_{C^1} \|\eta \|_{W^{1,p}}
\end{equation}

Inequalities \eqref{C0-control} and \eqref{C1-control} show that 
\[
\|F_{J'} - F_{J} \|_{C^1} \leq C \| J'-J\|_{C^1}
\]
where on the left-hand side the $C^1$-norm is for maps from the ball of radius $r$ in $W^{1,p}(u^*TZ)$ to $L^p(\Lambda^1 \otimes_\R u^*TZ)$. Moreover, the images of the maps are all closed linear subspaces of $L^p$. Now the fact that $F_J$ has a regular zero at the origin implies, by a standard implicit function theorem argument, that when $\|J'-J\|_{C^1}$ is small enough, $F_{J'}$ has a zero near the origin. 
\end{proof}

\subsection{Regularity of holomorphic curves in hyperkähler twistor spaces}\label{regularity-proof}

In this section we will show that a holomorphic 2-sphere in a hyperkähler 4-manifolds lifts to a \emph{regular} holomorphic curve in the twistor space.  To fix notation, we begin by recalling the twistor space $Z$ of a hyperkähler 4-manifold $X$. Write $I_1, I_2, I_3$ for a hyperkähler triple of complex structures on $X$ (satisfying the quaternion relations); then any other hyperkähler structure has the form $I_a =a_1I_1 + a_2I_2 + a_3I_3$ where $a= (a_1,a_2,a_3)$ is a point of the unit sphere $S^2 \subset \R^3$. The twistor space of $X$ is simply $Z = X \times S^2$ and the two twistor almost complex structures are defined as:
\begin{gather*}
T_{(x,a)}(X\times S^2) = T_xX \oplus T_aS^2\\
J_{\pm}(x,a) = \pm I_a \oplus J_{S^2}
\end{gather*}
Given an $I_a$-holomorphic curve $f \colon (\Sigma, j) \to (X,I_a)$ the lift to $Z$ given by $u(\sigma) = (f(\sigma), a)$ is a $J\pm$-holomorphic curve, called the \emph{horizontal lift} of $f$. 

For the rest of this section we fix a choice of $J_\pm$, and denote it by $J$; the arguments are insensitive to this choice. 

\begin{proposition}
Let $f \colon \Sigma \to X$ be a non-constant compact curve in a hyperkähler 4-manifold which is holomorphic for one of the hyperkähler complex structures, $I$. Write $u \colon \Sigma \to Z$ for the horizontal lift. The infinitesimal $J$-holomorphic deformations of $u$ in $Z$ are all given by infinitesimal $I$-holomorphic deformations of $f$ in $X$. 
\end{proposition}
\begin{proof}
Without loss of generality, we assume that $f$ is $I_3$-holomorphic. Identify the tangent space at $I_3$ to the hyperkähler sphere with the span of $I_1, I_2$ and write an arbitrary section $\xi \in C^\infty(u^*TZ)$ as  $\xi = (\eta, v)$ where $\eta \in C^{\infty}(f^*TX)$ and $v \colon \Sigma \to \Span(I_1,I_2)$. Suppose that $\xi$ is an infinitesimal $J$-holomorphic deformation of $u$. We first prove that the factor $v$ is constant. By definition of the twistorial almost complex structures, the map $p \colon (Z,J) \to S^2$ given by projection onto the second factor of $Z= X \times S^2$ is holomorphic. It follows that $p_*(\xi)$ is an infinitesimal deformation of $p \circ u$ as a holomorphic map $\Sigma \to S^2$. But $p \circ u$ is constant and, since the only holomorphic deformations are through constant maps, $v$ is constant. 

We next prove that $v$ actually vanishes, which will complete the proof of the proposition. Suppose for a contradiction that $v$ is non-zero. Without loss of generality, rescaling $v$ and rotating our axes, we can assume $v = I_2$. Let $J_t$ be a path of hyperkähler complex structures on $X$ with $J_0 = I_3$ and $J'(0) = I_2$. Then $\eta$ is tangent to a path of maps $f_t \colon \Sigma \to X$ which solve, to first order at least, the $J_t$-holomorphic curve equation:
\[
\diff f_t + J_t(f_t)\circ \diff f_t \circ j = 0
\]
Differentiating at $t=0$ and using the fact that $I_2 \circ \diff f \circ j = I_1 \circ \diff f$ (since $f$ is $I_3$-holomorphic) we find
\begin{equation}\label{dolbeault-trivial}
\delb \eta + I_1 \circ \diff f =0
\end{equation}

Away from the zeros of $\diff f$, we define the holomorphic normal bundle $N = f^*TZ/\im \diff f$. We now bring in the complex volume form $\Omega$ on $X$ which we interpret as a holomorphic $1$-form on $\Sigma$ with values in $N^*$. We project \eqref{dolbeault-trivial} onto $N$, pair with $\Omega$ and then integrate to obtain
\[
\int_\Sigma \delb (\pi_N\eta) \wedge \Omega + \pi_N(I_1 \circ \diff f) \wedge \Omega =0
\]
At first sight the integrand only makes sense, like the normal bundle, away from the zeros of $\diff f$, but \eqref{dolbeault-trivial} shows that at the zeros of $\diff f$, the form $\delb \eta$ also vanishes and so the integrand actually vanishes at the zeros of $\diff f$. Now integration by parts, together with Stokes's theorem and the fact that $\delb \Omega = 0$ implies that the first term integrates to zero leaving
\begin{equation}\label{area-zero}
\int_\Sigma \pi_N(I_1 \circ \diff f) \wedge \Omega = 0
\end{equation}

We now investigate this integrand. Work away from the zeros of $\diff f$ and choose holomorphic coordinates $z$ near $\sigma \in \Sigma$ and $I_3$-holomorphic coordinates $(z,w)$ near $f(\sigma) \in X$ such that $f(z) = (z,0)$. Write $z = x+iy$, $w= p+iq$. We can arrange things so that \emph{at the point $f(\sigma)$} we have that the hyperkähler structure is standard, identified with the quaternions $x+iy+jp+kq$, with $\Omega = \diff z \wedge \diff w$, $I_1$ corresponding to multiplication by $j$, $I_2$ to multiplication by $k$ and $I_3$ to multiplication by $i$. With this choice of coordinates, one checks that $I_1 \circ \diff f = \diff x \otimes \del_p - \diff y \otimes \del_q$ (at $\sigma$) and so
\[
 \pi_N(I_1 \circ \diff f) \wedge \Omega = 2 i \diff x \wedge \diff y = 2i \diff A
\]
 where $\diff A$ is the area form on $\Sigma$ induced by the pull-back of the Riemannian metric from $X$ via $f$. This gives us our contradiction, since \eqref{area-zero} now implies that the image of $f$ has zero area and so $f$ is constant, contrary to hypothesis.
\end{proof}

\begin{corollary}\label{spheres-are-regular}
Let $f \colon \Sigma \to X$ be a non-constant compact curve in a hyperkähler 4-manifold which is holomorphic for one of the hyperkähler complex structures. The horizontal lift $u$ to twistor space is regular for $J_+$ or $J_-$ if and only if $\Sigma$ has genus zero. 
\end{corollary}
\begin{proof}
We begin by computing the index of the linearised Cauchy--Riemann equations. Note that $c_1(Z,J_+)$ is pulled back from $S^2$ (see \cite{Atiyah:1978aa}), whilst $c_1(Z,J_-)=0$ (see \cite{Fine:2009aa}) and so in both cases $u^*c_1$ vanishes. This means that the (real) index is $6-6g$ where $g$ is the genus of $\Sigma$. Already when $g >1$ the fact that the index is negative means the curve cannot be regular. When $g=1$, the index vanishes, yet there are non-trivial deformations coming from reparametrisations of $u$ given by biholomorphisms of $\Sigma$. Finally, when $g=0$, we use the previous result, that the infinitesimal $J$-holomorphic deformations of $u$ in $Z$ are equal to the holomorphic deformations of $f$ in $X$ keeping the complex structure on $X$ fixed. Since the normal bundle of $f$ is $\mathcal O(-2)$ the curve is geometrically rigid and the only deformations come from reparametrisations, i.e., precomposing with elements of $\PSL(2,\C)$.  This is a space of real dimension 6, equal to the index and so we conclude that the linearised Cauchy--Riemann operator is surjective in this case. 
\end{proof}

\begin{remark}
One might hope to account for the zero or negative index in the cases $g \geq 1$ by allowing the complex structure on the domain to vary (since the index in these cases equals the dimension of the moduli space of complex structures on the curve). Unfortunately this won't help in general, since experience with K3 surfaces leads one to expect the curve downstairs $f \colon \Sigma \to X$ to have non-trivial holomorphic deformations. 
\end{remark}

\subsection{Completing the proof of Theorem~\ref{main-result}}\label{details}

We now put the pieces together to complete the proof of Theorem~\ref{main-result}. By hypothesis, $(M_i,g_i)$ is a sequence of Riemannian manifolds with tame twistor spaces, converging in the pointed $C^2$-topology to a locally hyperkähler limit $X$. We assume, by passing to a subsequence, that either all the $M_i$ have $J_+$ tamed, or all have $J_-$ tamed. We focus attention on this choice of twistor almost complex structure from now on. Assume for a contradiction that the universal cover $\widetilde{X}$ contains a 2-sphere $S \subset \widetilde{X}$ which is $I$-holomorphic for one of the hyperkähler complex structures $I$. Let $U \subset \widetilde{X}$ be a neighbourhood of $S$ and $\widetilde{K} \subset \widetilde{X}$ a compact set containing $U$. Write $K \subset X$ for the image of $\widetilde{K}$. Denote by $(Z, J, \omega) \to K$ and $(Z_i, J_{Z_i}, \omega_{Z_i}) \to M_i$ the twistor spaces with their almost complex structures and closed 2-forms. By Lemma~\ref{twistor-spaces-converge}, there are fibrewise diffeomorphisms $\phi_i \colon Z \to Z_i$ such that $J_i = \phi_i^*J_{Z_i}$ converges to $J$ in $C^1$. Moreover, if we put $\omega_i= \phi^*\omega_{Z_i}$ then $[\omega_i] = [\omega]$ and $J_i$ is tamed by $\omega_i$. Pulling back to the twistor space $\widetilde{Z} \to \widetilde{K}$, we get a sequence $(\widetilde{\omega}_i, \widetilde{J}_i)$ of tamed almost complex structures which converge to the twistor almost complex structure $\widetilde{J}$ of $\widetilde{Z}$. We also have $[\widetilde{\omega}_i] = [\widetilde{\omega}]$, where $\widetilde{\omega}$ is the standard closed 2-form on twistor space, in this case the pull-back to $\widetilde{Z}$ of the area form on $S^2$ by the projection $\widetilde{Z} \to S^2$ to the sphere of hyperkähler complex structures.

Write $S'$ for the horizontal lift of $S$ to $\widetilde{Z}$, which is a $J$-holomorphic curve. By Proposition~\ref{perturbing-regular-curves} and Corollary~\ref{spheres-are-regular}, for all large $i$ there is a $\widetilde{J}_i$-holomorphic sphere $S'_i\subset \widetilde{Z}$ which is homotopic to $S'$. On the one hand, $\int_{S'_i} \widetilde{\omega}_i >0$, since $\widetilde{\omega}_i$ tames $\widetilde{J}_i$. On the other hand, $\int_{S'_i} \widetilde{\omega}_i = \int_{S'} \widetilde{\omega} =0$, since $[S'_i] = [S']$ and $[\widetilde{\omega}_i] = [\widetilde{\omega}]$. This contradiction completes the proof.

\bibliographystyle{siam}
\bibliography{limits_bibliography}

\bigskip
{ \footnotesize
\textsc{J.~Fine, Département de mathématiques, Université libre de Bruxelles, Brussels 1050, Belgium}\par\nopagebreak
\textit{E-mail address:} \texttt{\href{mailto:joel.fine@ulb.ac.be}{joel.fine@ulb.ac.be}}
}

\end{document}